\documentclass[12pt]{amsart}
\usepackage{amscd,amsmath,amsthm,amssymb}
\usepackage{pstcol,pst-plot,pst-3d}
\usepackage{lmodern,pst-node}
\usepackage{multicol}
\usepackage{epic,eepic}
\usepackage{amsfonts,amssymb,amscd,amsmath,enumerate,verbatim}
\psset{unit=0.7cm,linewidth=0.8pt,arrowsize=2.5pt 4}

\newpsstyle{fatline}{linewidth=1.5pt}
\newpsstyle{fyp}{fillstyle=solid,fillcolor=verylight}
\definecolor{verylight}{gray}{0.97}
\definecolor{light}{gray}{0.9}
\definecolor{medium}{gray}{0.85}
\definecolor{dark}{gray}{0.6}

\def\frk{\mathfrak}               

\def\Phi{{\frk N}}

\def\al{\alpha}
\def\bl{\beta}

\def\opn#1#2{\def#1{\operatorname{#2}}} 
\opn\chara{char} \opn\length{\ell} \opn\pd{pd} \opn\rk{rk}
\opn\projdim{proj\,dim} \opn\injdim{inj\,dim} \opn\rank{rank}
\opn\depth{depth} \opn\grade{grade} \opn\height{height}
\opn\Det{Det}\opn\size{size}
\opn\embdim{emb\,dim} \opn\codim{codim}

\opn\Tr{Tr} \opn\bigrank{big\,rank}
\opn\superheight{superheight}\opn\lcm{lcm}
\opn\trdeg{tr\,deg}
\opn\reg{reg} \opn\lreg{lreg} \opn\ini{in} \opn\lpd{lpd}
\opn\size{size}\opn{\mult}{mult}
\opn{\Cl}{Cl}
\opn{\PF}{PF}
\opn{\RF}{RF}
\opn{\MC}{MC}
\opn{\Kos}{Kos}
\opn\div{div} \opn\Div{Div} \opn\cl{cl} \opn\Cl{Cl}
\opn\Spec{Spec} \opn\Supp{Supp} \opn\supp{supp} \opn\Sing{Sing}
\opn\Ass{Ass} \opn\Min{Min} \opn\cl{cl}
\opn\Ann{Ann} \opn\Rad{Rad} \opn\Soc{Soc}
\opn\Syz{Syz} \opn\Im{Im} \opn\Ker{Ker} \opn\Coker{Coker}
\opn\Am{Am} \opn\Hom{Hom} \opn\Tor{Tor} \opn\Ext{Ext}
\opn\End{End} \opn\Aut{Aut} \opn\id{id} \opn\ini{in}\opn\GCD{GCD}
\opn\nat{nat}
\opn\pff{pf}
\opn\Pf{Pf} \opn\GL{GL} \opn\SL{SL} \opn\mod{mod} \opn\ord{ord}
\opn\Gin{Gin}
\opn\Hilb{Hilb}\opn\adeg{adeg}\opn\std{std}\opn\ip{infpt}
\opn\Pol{Pol}
\opn\sat{sat}
\opn\Var{Var}
\opn\Gen{Gen}
\opn\lex{lex}
\opn\div{div}
\opn\NUF{NUF}
\opn\mNUF{mNUF}
\opn\type{type}

\opn\PF{PF}
\opn\Fr{F}
\opn\Ap{Ap}
\opn\aff{aff} \opn\con{conv} \opn\relint{relint} \opn\st{st}
\opn\lk{lk} \opn\cn{cn} \opn\core{core} \opn\vol{vol}
\opn\link{link} \opn\star{star}
\opn\gr{gr}

\def\pot#1#2{#1[\kern-0.28ex[#2]\kern-0.28ex]}

\opn\dirlim{\underrightarrow{\lim}}
\opn\inivlim{\underleftarrow{\lim}}
\let\to=\rightarrow

\def\Implies{\ifmmode\Longrightarrow \else
        \unskip${}\Longrightarrow{}$\ignorespaces\fi}
\def\implies{\ifmmode\Rightarrow \else
        \unskip${}\Rightarrow{}$\ignorespaces\fi}
\def\iff{\ifmmode\Longleftrightarrow \else
        \unskip${}\Longleftrightarrow{}$\ignorespaces\fi}

\let\:=\colon
\newtheorem{Theorem}{Theorem}[section]
\newtheorem{thm}{Theorem}[section]

\newtheorem{lem}[Theorem]{Lemma}

\newtheorem{cor}[Theorem]{Corollary}

\newtheorem{rem}[Theorem]{Remark}

\newtheorem{Claim}[Theorem]{Claim}
\newtheorem*{acknowledgement}{Acknowledgement}

\let\epsilon\varepsilon
\let\kappa=\varkappa
\def\qed{\ifhmode\textqed\fi
      \ifmmode\ifinner\quad\qedsymbol\else\dispqed\fi\fi}
\def\textqed{\unskip\nobreak\penalty50
       \hskip2em\hbox{}\nobreak\hfil\qedsymbol
       \parfillskip=0pt \finalhyphendemerits=0}
\def\dispqed{\rlap{\qquad\qedsymbol}}

\opn\dis{dis}
\def\pnt{{\raise0.5mm\hbox{\large\bf.}}}

\opn\Lex{Lex}
\opn\int{int}

\newcommand{\inD}[1][\relax]{\def\argone{#1}\def\temprelax{\relax}
  \ifx\argone\temprelax\right.\else\,\middle|#1\right.{}\fi}

\newif\ifbinary
\binarytrue
\newcommand{\Z}{\mathbb Z}

\begin{document}

\title{A short proof of Bresinski's Theorem on Gorenstein semigroup rings generated by 4 elements.}

\author{Kei-ichi Watanabe}
\thanks{(Kei-ichi Watanabe) Department of Mathematics, College of
Humanities and Sciences,
Nihon University, Setagaya-ku, Tokyo, 156-8550, Japan}
\thanks{This work was partially supported by JSPS  KAKENHI 
Grant Number  26400053}

\begin{abstract}
Let $H=\langle n_1,\ldots ,n_4\rangle$ be a  numerical semigroup generated by $4$ elements, 
which is symmetric and let $k[H]$ be the semigroup ring of $H$ over a field $k$.  
H. Bresinski proved in \cite{Br} that the defining ideal of $k[H]$ is minimally generated by 
$3$ or $5$ elements. We give a new short proof of Bresinski's Theorem using the structure 
theorem of Buchsbaum and Eisenbud on the  minimal free resolution of 
Gorenstein rings of embedding codimension $3$.
\end{abstract}

\maketitle

\section{Basic concepts}

Let $H=\langle n_1,\ldots ,n_4\rangle$ be a  numerical semigroup generated by $4$ elements.
We denote 
\[\Fr(H) = \max \{ n\in \Z\;|\; n\not\in H\}\]
 the Frobenius number of $H$ and $N = \sum_{i=1}^4 n_i$.   
We call $H$ symmetric if for every $n\in \Z$, $n\in H$ if and only if $\Fr(H) - n\not\in H$.
Let $k[H]$ be the semigroup ring of $H$ over a field $k$ and $S=k[x_1,\ldots,x_4]$ be the 
polynomial ring over $k$ in the indeterminates $x_1,\ldots,x_4$. 
It is known by \cite{Ku} that $H$ is symmetric if and only  $k[H]$ is Gorenstein.
Let $\pi: S\to k[H]$ be the surjective $k$-algebra homomorphism with $\pi(x_i)=t^{n_i}$ for $i=1,\ldots,n$.  
We consider $S$ as a graded ring putting $\deg(x_i) = n_i$ so that $\pi$ preserves the degree. 
We denote by $I_H$ the kernel of $\pi$. If we assign to each $x_i$ the degree $n_i$, then with respect to this grading, $I_H$ is a homogeneous ideal, generated by binomials. A binomial $\phi=\prod_{i=1}^ex_i^{\al_i}-\prod_{i=1}^ex_i^{\bl_i}$ belongs to $I_H$ if and only if $\sum_{i=1}^e\al_in_i=\sum_{i=1}^e\bl_in_i$. 

\bigskip
We define $\al_i$ to be the minimal positive integer such that
\begin{eqnarray}
\label{minimal}
\al_i n_i = \sum_{j=1, j\ne i}^4 \al_{ij} n_j.
\end{eqnarray}

Thus $f_i = x_i^{\al_i} - \prod_{j=1, j\ne i}^4 x_j^{\al_{ij}}$ ($i=1,2,3,4$) is a minimal generator 
of $I_H$. 

\bigskip
The purpose of this note is to give a short proof of Bresinski's Theorem;

\begin{thm}\label{Br}
Assume that $H$ is symmetric generated by $4$ elements. If $k[H]$ is not a complete 
intersection, then $I_H$ is minimally generated by $5$ elements.
\end{thm}

For the proof we let 
\[F_{\bullet} =[\; 0\to F_3 \to F_2 \overset{d_2}{\to} F_1 \overset{d_1}{\to} F_0 = k[H]\to 0\; ]\] 
be the graded minimal free resolution of $k[H]$ over $S$.  
Note that \lq\lq $H$ is symmetric" is equivalent to say \lq\lq $k[H]$ is a Gorenstein ring".
We denote $r = \mu(I_H) = \rank F_1$, $\phi_1, \ldots \phi_r$ be free basis of $F_1$ 
and we  put $f_i = d_1(\phi_i)\in I_H$. We always assume that each $f_i$ is a binomial.

Let us summarize known results about $F_{\bullet}$.

\begin{thm}\label{BE} (\cite{WJ}, \cite{BE}) (1) Since $k[H]$ is Gorenstein with $a$-invariant $a(k[H])= \Fr(H)$, 
$F_3 \cong S(-\Fr(H) - N)$ and $F_{\bullet}$ is self-dual in the sense there is an 
isomorphism $\Hom_S( F_{\bullet} , F_3) \cong F_{\bullet}$. \par
(2)   $r$ is an odd number.\par  
(3)  Let $M=(m_{ij})$ be the $r$ by $r$ matrix corresponding $d_2: F_2\to F_1$. Then 
we can choose the bases of $F_2$ and $F_1$ so that $M$ is a skew-symmetric matrix. \par
(4) Let  $\{ e_1,\ldots , e_r\}$ be the free basis of $F_2$ so that  $d_2(e_i) = 
\sum_{j=1}^r m_{ij} \phi_j$. Then  if $M(i)$ denotes the $(r-1)\times (r-1)$ matrix obtained by deleting $i$-th 
row and $i$-th column of $M$, then $f_i$ is obtained as the Pfaffian of $M(i)$ and $\deg(e_i) = 
\Fr(H) + N -\deg(f_i)$.  Namely, $\Det(M(i)) = f_i^2$. 

Note that if the $i$-th row of $M$ is $(m_{i1}, \ldots , m_{ir})$, then we have
\[(*) \quad \sum_{i=1}^r m_{ij}f_j =0.\]
\end{thm}

\section{The proof.}

Now we will give a proof of Theorem \ref{Br} using Theorem \ref{BE}. 

Renumbering $\{ f_1,\ldots , f_r\}$, we  can assume
$f_p = x_p^{\al_p} - q_p$, where $q_p$ is a monomial of $\{x_1,\ldots , x_4\}\setminus \{x_p\}$ 
($p= 1,\ldots , 4$).  

Hence, for $p\ge 5$, $f_p$ is of the form 

\[(**) \quad f_p= x_i^ax_j^b - x_k^cx_l^d  \quad (a,b,c,d >0)\quad (p\ge 5)\]

for some permutation $\{x_i,x_j,   x_k, x_l\}$ of $\{x_1,x_2,   x_3, x_4\}$.

\begin{rem}  In the setting above, we assumed $f_p= - f_q$ ($1\le p< q \le 4$) does not occur.
But the argument below shows that if $I_H$ has a minimal generator of the form 
 $x_p^{\al_p} - x_q^{\al_q}=0$, then we get a  contradiction easier than the following argument
 as is explained in Remark \ref{fp=fq} .
\end{rem}

Now we will show that $r =5$. So, we assume  $r\ge 7$ and get a contradiction.

\begin{lem}\label{s,t ge 5} If $s,t\ge 5$ and $s\ne t$, then $m_{s,t}=0$. 
\end{lem}
\begin{proof} Assume $m_{s,t}\ne 0$ with $\deg(m_{s,t}) =h\in H_+$.
Then we will have 
\[\deg e_s = \Fr(H) + N - \deg(f_s) = h + \deg(f_t)\]
or  
\[ \Fr(H) + N = h + \deg f_s + \deg f_t.\]

Since $s,t\ge 5$, $f_s, f_t$ are of the form (**) and we can take the expression 
\[h + \deg f_s = \sum_{i=1}^4 a_i n_i\]
so that $3$ $a_i$'s among $4$ are positive.  If some $a_j=0$, then we can 
choose expression of $\deg f_t$ so that the coefficient of $n_j$ is positive.  

That means $h + \deg f_s + \deg f_t \ge_H N$, 
where we denote $a\ge_H b$ if $a-b\in H$. Then we get $\Fr(H)\in H$, a 
contradiction !
\end{proof}

\begin{cor}  $r\le 7$.
\end{cor} 
\begin{proof} Assume $r\ge 9$.  We know that $f_1$ is the Pfaffian of the matrix 
$M(1)$.   Then by Lemma \ref{s,t ge 5}, we can see $\Det(M(1))=0$ because 
$m_{s,t}=0$ if $s,t\ge 5$,
\end{proof}

\begin{rem}\label{fp=fq}  Assume $I_H$ has a minimal generator of the form $x_p^{\al_p} - x_q^{\al_q}=0$.
Then we can assume $f_p$ for $p\ge 4$ is of the form (**).  Then above argument shows that 
$m_{s,t} = 0$ for $s,t\ge 4$.  Now, if $r=7$, then $\Det(M(1))=0$ since it contains a $4\times 4$ 
$0$ matrix in it. Thus to show $r\le 5$ we can assume  there is no minimal generator of $I_H$ of 
type $x_p^{\al_p} - x_q^{\al_q}$.
\end{rem}

Let us  continue the proof of the Theorem.  
We  assume $r=7$ and deduce a contradiction.

We must have $f_1 = x_1^{\al_1} - p_1$ as the Pfaffian of $M(1)$.
Now we know by Lemma \ref{s,t ge 5} that if $s,t \ge 5$, $m_{s,t}=0$.  
Let $N(1)$ be $3 \times 3$ matrix which is $2-4$ rows and $5-7$ columns of $M$.
Then  we must have $\Det(N(1)) = \pm x_1^{\al_1} - p_1$. 
That means, for every $s$, $2\le s \le 4$, there should exist $t$ with $5\le t\le 7$ such that 
$m_{s,t}$ is a power of $x_1$.   Namely, there should be at least $3$ components 
that are a power of $x_1$. 

Since the same should be true for 
$x_2,\ldots , x_4$, there should be $3\times 4 =12$ components in $1-4$ rows 
and $5-7$ columns.  Namely we get

\begin{Claim}\label{(*)} Every $(s,t)$ component of $M$ with 
$1\le s\le 4$ and $5 \le t\le 7$ is a power of some $x_i$ and consequently $\ne 0$. 
\end{Claim}

On the other hand, assume, say, $f_1 = x_1^{\al_1}  - x_2^bx_3^c$ with $b,c > 0$ 
and also $f_t = x_1^ax_4^d - x_2^{b'}x_3^{c'}$ for some $a,b',c',d>0$ and $5\le t \le 7$.
Then $m_{1,t}$ should be $0$, since otherwise
\[ \Fr(H) + N = h +  \deg f_1+ \deg f_t = h + an_1+ b' n_2+ c'n_3 + dn_4 \ge _H N, \]    
 which will lead to $\Fr(H)\in H$.  A contradiction!  Hence Claim \ref{(*)} will 
 lead to a contradiction.
 
 Hence we get a contradiction from $r=7$ and hence $\mu(I_H) = 5$ if $H$ is not a 
 complete intersection.
 
\begin{rem}  If $r=5$, we can show that $I_H$ has no minimal generator of the form 
$x_p^{\al_p} - x_q^{\al_q}=0$.  
\end{rem}
\begin{proof}
Assume that $f_1= x_1^a-x_2^b, f_2= x_3^c- q_3,
f_3= x_4^d- q_4$ for some monomials $q_3,q_4$ and $f_4,f_5$ are of the form (**).
Then above argument shows $m_{4,5}=m_{5,4}=0$ and $\Det(M(1)) = (x_1^a-x_2^b)^2$.
That means, $m_{i,j}$ are some power of $x_1$ or $x_2$ for $(i,j) = (2,4), (2,5), (3,4), (3,5)$.
Then it is easy to see it is impossible to get a power of $x_3$ in $\Det(M(2))$, which 
contradicts $\Det(M(2))  = (x_3^c- q_3)^2$. 
\end{proof}

\begin{cor}  If $I_H$ has an element of type $x_p^{\al_p} - x_q^{\al_q}=0$ as a minimal generator,
then $H$ is a complete intersection. 
\end{cor}

\begin{rem}  If $r=5$, we can deduce the form of $M$ by our argument.  
For a monomial $m$ of $\{x_1, \ldots , x_4\}$, let us put 
 \[\supp(m) = \{ x_i \;| x_i \;{\mbox{\rm divides}}\;  m\}.\]
Then, if we put $f_i = x_i^{\al_i} - q_i$
($i=1,\ldots , 4$) and $f_5 = q_5 - q_6$, then we can show that $\supp(q_i)$ ($i=1,\ldots , 6$)
are all different and $\supp(q_i)$ consists of $2$ variables.   
Also, if $f_i = x_i^{\al_i} -q_i, f_j = x_j^{\al_j} -q_j$ with $\supp(q_i) \cup \supp(q_j) =\{ x_1, \ldots , x_4\}$,
then $m_{i,j} = 0 = m_{j,i}$ and these are the only $0$ of $M$ except diagonals. 
Thus we have exactly $16 = 5^2 - 5 -4$ non-$0$ entries of $M$ and they are powers of 
some $x_i$.    Thus we could deduce  the matrix 
\[\left(\begin{array}{ccccc} 0 & -x_3^{\al_{43}} & 0 & -x_2^{\al_{32}} & - x_4^{\al_{24}}\\
x_3^{\al_{43}} & 0 & x_4^{\al_{14}} & 0 & -x_1^{\al_{31}}\\
0 &  -x_4^{\al_{14}} & 0 & -x_1^{\al_{21}} & - x_2^{\al_{42}}\\
x_2^{\al_{32}} & 0 & x_1^{\al_{21}} & 0 & -x_3^{\al_{13}} \\
x_4^{\al_{24}} & x_1^{\al_{31}} & x_2^{\al_{42}} & x_3^{\al_{13}} & 0
\end{array}\right)
   \]
 in Theorem 4 of \cite{BFS}.
\end{rem}
\begin{proof}  Since  $\Det(M(1))= (x_1^{\al_1}- q_1)^2$, 
 there should be at least a power of $x_1$ in $2,3,4,5$ rows. 
Since  $\Det(M(i))= (x_i^{\al_i}- q_i)^2$ for $i=2,3,4$, 
 $m_{i5}$ is a power of some $x_k$, $k\ne i$ and 
a power of every $x_k$ ($1\le k\le 4$) should appear as some $m_{5i}$.  Also we have 
\[(**) \qquad \sum_{1=1}^4 m_{i5} f_i =0.\]
If, say, $\supp(q_1)$ has $3$ variables, then we will have $m_{15}=0$, since then 
we will have $\deg f_1+ \deg f_5\ge_H N$.  Then we must have $m_{51}=0$, 
contradicting  observation above. 
Thus we know that every $q_i$ contains exactly $2$ variables. \par

From (**) we know that if $m_{15} = x_j^p$, then $q_1$ is of the form $q_1= x_j^s q_1'$, 
where we must have $\al_j = p+s$. 
Thus changing the order of variables, if necessary,  we may assume 
$f_1= x_1^{\al_1} - x_3^{\al_{13}}x_4^{\al_{14}}$ and $m_{51}= x_4^p$.
Then we have $p+ \al_{14} = \al_4$  since $m_{54}x_4^{\al_4}$ must cancel with $x_4^px_3^{\al_{13}}x_4^{\al_{14}}$ and  we must have $m_{54} = x_3^{\al_{13}}$. 
Then $ x_3^{\al_{13}}q_4$ must cancel with  $m_{53}x_3^{\al_3}$. Thus we have
\[q_1=x_3^{\al_{13}}x_4^{\al_{14}}, q_2= x_a^{\al_{21}}x_4^{\al_{24}}, q_3=x_1^{\al_{31}}x_2^{\al_{32}},
q_4= x_2^{\al_{42}}x_3^{\al_{43}}\]
and $m_{51}= x_4^{\al_{24}}, m_{52}= x_1^{\al_{31}}, m_{53}= x_2^{\al_{42}}, m_{54}=x_3^{\al_{13}}$
with $\al_1=\al_{24}+\al_{21}, \al_2=\al_{32}+\al_{42}, \al_3= \al_{13}+\al_{43}, \al_4= \al_{24}+\al_{14}$.
We notice that $m_{13}=m_{24}=0$ since $\deg f_1+\deg f_3, \deg f_2 + \deg f_4\ge_H N$ and 
we can fill in the other parts of $M$ by $m_{i5} = m_{5i}$ and $\sum m_{ij} f_j =0$.
\end{proof}

\begin{acknowledgement} The author thanks Kazufumi Eto for bringing him to this subject 
by his inspiring talk at \lq\lq Singularity Seminar" at Nihon University.  The main technique 
of this paper came out in the collaboration with J\" urgen Herzog and the author is grateful 
to him for the collaboration.
\end{acknowledgement}

\end{document}